\documentclass{amsart}
\usepackage{latexsym}
\usepackage{amstext}
\usepackage{amsmath}
\usepackage{amssymb}
\usepackage{amsthm}
\usepackage{mathrsfs}
\usepackage{url}
\usepackage{enumerate}

\newtheorem{theorem}{Theorem}

\newtheorem{corollary}[theorem]{Corollary}

\theoremstyle{definition}

\newtheorem{question}[theorem]{Question}

\def \P {\mathbb{P}}

\def \I {\mathcal{I}}

\DeclareMathOperator{\Supp}{Supp}

\begin{document}
\bibliographystyle{amsplain}
\title{On the $p$-adic Second Main Theorem}
\author{Aaron Levin}
\address{Department of Mathematics\\Michigan State University\\East Lansing, MI 48824}
\email{adlevin@math.msu.edu}
\date{}

\begin{abstract}
We study the Second Main Theorem in non-archimedean Nevanlinna theory, giving an improvement to the non-archimedean Second Main Theorems of Ru and An in the case where all the hypersurfaces have degree greater than one and all intersections are transverse.  In particular, under a transversality assumption, if $f$ is a nonconstant non-archimedean analytic map to $\mathbb{P}^n$ and $D_1,\ldots, D_q$ are hypersurfaces of degree $d$, we prove the defect relation
\begin{equation*}
\sum_{i=1}^q\delta_f(D_i)\leq n-1+\frac{1}{d},
\end{equation*}
which is sharp for all positive integers $n$ and $d$.
\end{abstract}

\maketitle

\section{Introduction}

Classical Nevanlinna theory involves the study of the distribution of values of meromorphic functions.  It may be viewed as giving a quantitative generalization of Picard's theorem that a nonconstant entire function omits at most one complex value.  The foundation of the theory consists of two aptly named fundamental results:  the First Main Theorem, which is a straightforward consequence of Jensen's formula, and the deeper Second Main Theorem.  In higher-dimensional Nevanlinna theory, where one studies, for example, holomorphic maps from the complex numbers to complex projective varieties, the First Main Theorem readily extends while generalizations of the Second Main Theorem remain largely conjectural.

In analogy with the theory for complex numbers, a version of Nevanlinna theory has been constructed for $p$-adic meromorphic functions on $\mathbb{C}_p$, the completion of the algebraic closure of the $p$-adic numbers $\mathbb{Q}_p$.  More generally, the theory can be developed for any algebraically closed field $K$ complete with respect to a non-archimedean absolute value.

In this setting, for analytic maps to projective space, the First Main Theorem and the Second Main Theorem (due to Ru \cite{Ru}) take the following form (see Section~\ref{snot} for the definitions):

\begin{theorem}[First Main Theorem]
Let $D$ be a hypersurface in $\mathbb{P}^N$ over $K$ of degree $d$.  Let $f:K\to \mathbb{P}^N$ be a nonconstant analytic map with $f(K)\not\subset D$.  Then for all $r>0$,
\begin{equation*}
m_f(r,D)+N_f(r,D)=dT_f(r)+O(1).
\end{equation*}
\end{theorem}

\begin{theorem}[Second Main Theorem (Ru \cite{Ru})]
Let $D_1,\ldots, D_q$ be hypersurfaces in $\mathbb{P}^n$ over $K$ in general position.  Let $f:K\to \mathbb{P}^n$ be a nonconstant analytic map whose image is not completely contained in any of the hypersurfaces $D_1,\ldots, D_q$.  Then for all $r\geq 1$,
\begin{equation*}
\sum_{i=1}^q \frac{m_f(r,D_i)}{\deg D_i}\leq  nT_f(r)+O(1).
\end{equation*}
\end{theorem}

Here, and throughout the paper, the implied constant in the $O(1)$ is independent of $r$.  In the case where all the hypersurfaces are hyperplanes, a Second Main Theorem with a ramification term was proven by Boutabaa in his Ph.D. thesis \cite{Bou} (see also \cite{Bou2} and the papers of Khoai and Tu \cite{KT} and Cherry and Ye \cite{CY}).

More generally, An \cite{An} proved a Second Main Theorem for analytic maps to projective varieties.

\begin{theorem}[Second Main Theorem (An \cite{An})]
\label{SMTAn}
Let $X\subset \mathbb{P}^N$ be a projective variety over $K$ of dimension $n\geq 1$.  Let $D_1,\ldots, D_q$ be hypersurfaces in $\mathbb{P}^N$ over $K$ in general position with $X$.  Let $f:K\to X$ be a nonconstant analytic map whose image is not completely contained in any of the hypersurfaces $D_1,\ldots, D_q$.  Then for all $r\geq 1$,
\begin{equation*}
\sum_{i=1}^q \frac{m_f(r,D_i)}{\deg D_i}\leq  nT_f(r)+O(1).
\end{equation*}
\end{theorem}

The purpose of this paper is to give an improvement to this Second Main Theorem when none of the hypersurfaces are hyperplanes and all intersections are transverse.

\begin{theorem}
\label{tmain2}
Let $X\subset \mathbb{P}^N$ be a projective variety over $K$ of dimension $n\geq 1$.  Let $D_1,\ldots, D_q$ be hypersurfaces in $\mathbb{P}^N$ over $K$ such that all intersections amongst $D_1,\ldots, D_q$, and $X$ are transverse.  Let $f:K\to X$ be a nonconstant analytic map whose image is not completely contained in any of the hypersurfaces $D_1,\ldots, D_q$.  Then for all $r\geq 1$,
\begin{equation*}
\sum_{i=1}^q \frac{m_f(r,D_i)}{\deg D_i}\leq  \left(n-1+\max_i \frac{1}{\deg D_i}\right)T_f(r)+O(1).
\end{equation*}
In particular, if $\deg D_i=d$ for all $i$, then for all $r\geq 1$,
\begin{equation}
\label{shin}
\frac{1}{d}\sum_{i=1}^q m_f(r,D_i)\leq  \left(n-1+\frac{1}{d}\right)T_f(r)+O(1).
\end{equation}

\end{theorem}

More generally, we prove (Theorem \ref{tmain}) an inequality in terms of an invariant related to the intersections between $D_1,\ldots, D_q$, and $X$.  In Theorem \ref{tsharp}, we will show that the inequality \eqref{shin} is sharp for $X=\mathbb{P}^n$ and all values of $n$ and $d$.

As a consequence of Theorem \ref{tmain2}, in Section \ref{SSMT} we recover a result of An, Wang, and Wong.

\begin{corollary}[An, Wang, Wong \cite{AWW}]
\label{AWW}
Let $D_1,\ldots, D_n$ be hypersurfaces in $\mathbb{P}^n$ intersecting transversally.  Then $\P^n\setminus\cup_{i=1}^nD_i$ is $K$-hyperbolic if $\deg D_i\geq 2$ for $1\leq i\leq n$.
\end{corollary}

Thus, Theorem \ref{tmain2} may be viewed as giving a quantitative generalization of An, Wang, and Wong's result.

If we define the defect of $f$, with respect to a hypersurface $D$, by
\begin{equation*}
\delta_f(D)=\liminf_{r\to\infty} \frac{m_f(r,D)}{(\deg D)T_f(r)},
\end{equation*}
then Theorem \ref{tmain2} immediately implies the defect relations:
\begin{corollary}
Under the same hypotheses as Theorem {\rm \ref{tmain2}},
\begin{equation*}
\sum_{i=1}^q\delta_f(D_i)\leq n-1+\max_i \frac{1}{\deg D_i}.
\end{equation*}
In particular, if $\deg D_i=d$ for all $i$, then
\begin{equation*}
\sum_{i=1}^q\delta_f(D_i)\leq n-1+\frac{1}{d}.
\end{equation*}
\end{corollary}

Again, by Theorem \ref{tsharp}, this last defect relation is sharp in the case of projective space.

We conclude this section with a comparison to results in classical Nevanlinna theory.  The counterpart to Ru's non-archimedean Second Main Theorem is a result of Er{\"e}menko and Sodin \cite{ES}.
\begin{theorem}[Er{\"e}menko, Sodin \cite{ES}]
\label{tES}
Let $D_1,\ldots, D_q$ be complex hypersurfaces in $\mathbb{P}^n$ in general position.  Let $f:\mathbb{C}\to \mathbb{P}^n$ be a nonconstant holomorphic map whose image is not completely contained in any of the hypersurfaces $D_1,\ldots, D_q$.  Then for any $\epsilon>0$,
\begin{equation*}
\sum_{i=1}^q \frac{m_f(r,D_i)}{\deg D_i}\leq  (2n+\epsilon)T_f(r)
\end{equation*}
for all $r>0$ outside a set of finite Lebesgue measure.
\end{theorem}

In the case where all of the hypersurfaces are hyperplanes this is a classical result of Cartan \cite{Ca}.  In view of Theorem \ref{tmain2}, it is natural to ask whether a corresponding improvement to the Er{\"e}menko-Sodin theorem is possible.

\begin{question}
Under the additional assumption that all intersections amongst the hypersurfaces are transverse, can the inequality in Theorem \ref{tES} be improved to
\begin{equation*}
\sum_{i=1}^q \frac{m_f(r,D_i)}{\deg D_i}\leq  \left(2\left(n-1+\max_i\frac{1}{\deg D_i}\right)+\epsilon\right)T_f(r)?
\end{equation*}
\end{question}

Finally, we mention a theorem of Ru \cite{Ru2} in the case where $f$ is assumed to be algebraically nondegenerate.

\begin{theorem}[Ru \cite{Ru2}]
Let $X\subset \mathbb{P}^N$ be a smooth complex projective variety of dimension $n\geq 1$.  Let $D_1,\ldots, D_q$ be hypersurfaces in $\mathbb{P}^N$ in general position with $X$.  Let $f:\mathbb{C}\to X$ be an algebraically nondegenerate holomorphic map.  Then for any $\epsilon>0$,
\begin{equation*}
\sum_{i=1}^q \frac{m_f(r,D_i)}{\deg D_i}\leq  (n+1+\epsilon)T_f(r)
\end{equation*}
for all $r>0$ outside a set of finite Lebesgue measure.
\end{theorem}

When $X=\mathbb{P}^n$, this result was proved in an earlier paper of Ru \cite{Ru3}.  When $X=\mathbb{P}^n$ and all of the hypersurfaces are hyperplanes, this is contained in the well-known Second Main Theorem of Cartan \cite{Ca}.

\section{Notation and definitions}
\label{snot}
Let $K$ be an algebraically closed field, of arbitrary characteristic, complete with respect to a non-archimedean absolute value $|\ |$.
Let $f(z)=\sum_{i=0}^\infty a_iz^i$ be an entire function on $K$.  For every real number $r\geq 0$, define
\begin{equation*}
|f|_r=\sup_i |a_i|r^i=\sup\{|f(z)|\mid z\in K, |z|\leq r\}=\sup\{|f(z)|\mid z\in K, |z|=r\}.
\end{equation*}

Let $f:K\to \P^N$ be a nonconstant analytic map.  Then we may write $f=(f_0,\ldots, f_N)$ where $f_0,\ldots, f_N$ are entire functions without a common zero.  We define the characteristic function $T_f(r)$ by
\begin{equation*}
T_f(r)=\|f\|_r:=\log\max\{|f_0|_r,\ldots,|f_N|_r\}.
\end{equation*}
Then $T_f(r)$ is well-defined, independent of the representation of $f$, up to an additive constant.

Let $D$ be a hypersurface in $\mathbb{P}^N$ defined by a homogeneous polynomial $Q\in K[x_0,\ldots, x_N]$ of degree $d$.  Consider the entire function $Q\circ f=Q(f_0,\ldots, f_N)$ on $K$, and assume from now on that $Q\circ f\not \equiv 0$, i.e., that $Q\circ f$ is not identically zero.  Let $n_f(r,D)$ be the number of zeros of $Q\circ f$ in the disk $B[r]=\{z\in K\mid |z|\leq r\}$, counting multiplicity.  This is independent of the choice of the defining polynomial $Q$ and the representation of $f$.  For $r>0$, we define the counting function
\begin{equation*}
N_f(r,D)=\int_0^r \frac{n_f(t,D)-n_f(0,D)}{t}dt+n_f(0,D)\log r
\end{equation*}
and the proximity function
\begin{equation*}
m_f(r,D)=\log \frac{\| f\|_r^d}{|Q\circ f|_r}.
\end{equation*}
Note that $m_f(r,D)$ depends on the choice of $Q$ only up to an additive constant.  A variety $X$ over $K$ is said to be $K$-hyperbolic if every analytic map from $K$ to $X$ is constant.

Let $D_1,\ldots, D_q$ be hypersurfaces in $\mathbb{P}^N$ over $K$.  We say that $D_1,\ldots, D_q$ are in general position if $\dim \cap_{i\in I}D_i\leq N-|I|$ for every subset $I\subset \{1,\ldots, q\}$ of cardinality $|I|\leq N+1$, where we set $\dim \emptyset =-1$.  If $X\subset \mathbb{P}^N$ is a projective variety over $K$ of dimension $n$, then we say that $D_1,\ldots, D_q$ are in general position with $X$ if $\dim \cap_{i\in I}D_i\cap X\leq n-|I|$ for every subset $I\subset \{1,\ldots, q\}$ of cardinality $|I|\leq n+1$.

Let $Y$ and $Z$ be closed subschemes of $\mathbb{P}^N$ with corresponding ideal sheaves $\I_Y$ and $\I_Z$, respectively.  We define $Y\subset Z$ if $\I_Z\subset \I_Y$.  We let $Y\cap Z$ denote the closed subscheme of $\mathbb{P}^N$ with ideal sheaf $\I_Y+\I_Z$.  If $M$ is a positive integer, we let $MZ$ be the closed subscheme of $\mathbb{P}^N$ with ideal sheaf $\I_Z^M$.  We let $\Supp Z$ denote the closed subscheme of $\mathbb{P}^N$ with the same underlying set of points as $Z$ and the reduced induced scheme structure.  By Hilbert's Nullstellensatz, for any closed subscheme $Z$, $Z\subset M\Supp Z$ for some positive integer $M$.

\section{A Second Main Theorem}
\label{SSMT}

We now state and prove our main theorem.

\begin{theorem}
\label{tmain}
Let $X\subset \mathbb{P}^N$ be a projective variety over $K$ of dimension $n\geq 1$.  Let $D_1,\ldots, D_q$ be hypersurfaces in $\mathbb{P}^N$ over $K$ that are in general position with $X$.  Let $M$ be a positive integer such that for any subset $I\subset \{1,\ldots, q\}$ of cardinality $|I|=n$, 
\begin{equation*}
\bigcap_{i\in I}D_i\cap X\subset M\Supp \left(\bigcap_{i\in I}D_i\cap X\right),
\end{equation*}
where we view $D_1,\ldots, D_q$, and $X$ as closed subschemes of $\mathbb{P}^N$.  Let $f:K\to X$ be a nonconstant analytic map whose image is not completely contained in any of the hypersurfaces $D_1,\ldots, D_q$.  Then for all $r\geq 1$,
\begin{equation*}
\sum_{i=1}^q \frac{m_f(r,D_i)}{\deg D_i}\leq  \left(n-1+\max_i \frac{M}{\deg D_i}\right)T_f(r)+O(1).
\end{equation*}
\end{theorem}

If all intersections amongst $D_1,\ldots, D_q$, and $X$ are transverse, then we may take $M=1$ in Theorem \ref{tmain}.  Thus, Theorem \ref{tmain2} from the Introduction follows immediately.

\begin{proof}
If $\max_i\frac{M}{\deg D_i}\geq 1$, then the theorem follows from Theorem \ref{SMTAn}.  So we may assume that $\max_i\frac{M}{\deg D_i}< 1$.
Fix $r\geq 1$.  Let $f=(f_0,\ldots, f_N)$, where $f_0,\ldots, f_N$ are entire functions without a common zero.  After reindexing, we can assume that
\begin{equation*}
m_f(r,D_1)\geq m_f(r,D_2)\geq \cdots \geq m_f(r,D_q).
\end{equation*}
Let $D_1,\ldots, D_q$ be defined by homogeneous polynomials $Q_1,\ldots, Q_q\in K[x_0,\ldots, x_N]$.  We first claim that there is a constant $C$, depending only on $Q_1,\ldots, Q_q$, and $X$, such that 
\begin{equation*}
m_f(r,D_i)\leq C, \quad i=n+1,\ldots, q.
\end{equation*}
This fact is used in the proofs of the Second Main Theorems of An and Ru, but we recall the argument for completeness.  Since the divisors $D_i$ are in general position with $X$, by Hilbert's Nullstellensatz, applied to $Q_1,\ldots, Q_{n+1}$ and the defining polynomials of $X$, we have
\begin{equation*}
x_j^{m_j}=\sum_{i=1}^{n+1}a_{ij}(x_0,\ldots, x_N)Q_i(x_0,\ldots, x_N)\quad \text{on $X$}, \quad j=0,\ldots, N,
\end{equation*}
for some positive integers $m_j$ and some homogeneous polynomials $a_{ij}\in K[x_0,\ldots, x_N]$, $\deg a_{ij}=m_j-\deg D_i$, $1\leq i\leq n+1, 0\leq j\leq N$.  Composing with $f$, this implies that
\begin{equation*}
|f_j|_r^{m_j}\leq C'\max_{1\leq i\leq n+1} |f|_r^{m_j-\deg D_i}|Q_i\circ f|_r,\quad j=0,\ldots, N,
\end{equation*}
where $C'$ is the maximum of the absolute values of the coefficients of the polynomials $a_{ij}$, $1\leq i\leq n+1, 0\leq j\leq N$.  Choosing $j$ so that $|f_j|_r=|f|_r$ and canceling $|f|_r^{m_j}$ from both sides gives
\begin{align*}
1&\leq C' \max_{1\leq i\leq n+1} \frac{|Q_i\circ f|_r}{|f|_r^{\deg D_i}}\\
&\leq C' \frac{|Q_{n+1}\circ f|_r}{|f|_r^{\deg D_{n+1}}}.
\end{align*}
Then $m_f(r,D_i)\leq \log C'$ for all $i>n$, as claimed.

We now handle $m_f(r,D_n)$.   Let
\begin{equation*}
(X\cap D_1\cap\cdots \cap D_n)(K)=\{P_1,\ldots, P_t\}.
\end{equation*}
Let $H_1,\ldots, H_t$ be hyperplanes through $P_1,\ldots, P_t$, respectively, such that $D_1,\ldots, D_{n-1}$, $H_1,\ldots, H_t$ are in general position with $X$ and the image of $f$ is not completely contained in any of the hyperplanes $H_1,\ldots, H_t$.  Let $H_i$ be defined by a linear form $L_i\in K[x_0,\ldots, x_N]$, $i=1,\ldots, t$.  Since 
\begin{equation*}
X\cap D_1\cap\cdots \cap D_n \subset M\Supp (X\cap D_1\cap\cdots \cap D_n),
\end{equation*}
it follows that
\begin{equation*}
(L_1\cdots L_t)^M=\sum_{i=1}^nb_iQ_i \quad \text{on $X$}
\end{equation*}
for some homogeneous polynomials $b_i\in K[x_0,\ldots, x_N]$, $\deg b_i=Mt-\deg D_i$, $i=1,\ldots, n$.  Then as before,
\begin{align*}
|L_1\circ f|_r^M\cdots |L_t\circ f|_r^M&\leq C''\max_{1\leq i\leq n}|f|_r^{Mt-\deg D_i}|Q_i\circ f|_r,\\
&\leq C''|f|_r^{Mt}\frac{|Q_n\circ f|_r}{|f|_r^{\deg D_n}},
\end{align*}
for some constant $C''$ independent of $r$.  Dividing both sides by $|f|_r^{Mt}$ and taking logarithms, we get
\begin{equation*}
m_f(r,D_n)\leq M\sum_{i=1}^t m_f(r,H_i)+\log C''.
\end{equation*}

It follows that
\begin{equation*}
\sum_{i=1}^q \frac{m_f(r,D_i)}{\deg D_i}\leq  \sum_{i=1}^{n-1}\frac{m_f(r,D_i)}{\deg D_i}+\frac{M}{\deg D_n}\sum_{i=1}^tm_f(r,H_i)+O(1).
\end{equation*}
Since $D_1,\ldots, D_{n-1},H_1,\ldots, H_t$ are in general position with $X$, by the same argument that was given at the beginning of the proof, all but $n$ of the proximity functions $m_f(r,D_1),\ldots, m_f(r,D_{n-1}), m_f(r,H_1),\ldots, m_f(r,H_t)$ are bounded by some constant not depending on $r$.  By the First Main Theorem,
\begin{align*}
\frac{m_f(r,D_i)}{\deg D_i}&\leq T_f(r)+O(1), \quad &&i=1,\ldots, n-1,\\	
m_f(r,H_i)&\leq T_f(r)+O(1), \quad &&i=1,\ldots, t.
\end{align*}
Since we have assumed that $\max_i \frac{M}{\deg D_i}<1$, this implies that,
\begin{equation*}
\sum_{i=1}^q \frac{m_f(r,D_i)}{\deg D_i}\leq  (n-1)T_f(r)+\max_i \frac{M}{\deg D_i}T_f(r)+O(1).
\end{equation*}

\end{proof}

As a consequence of Theorem \ref{tmain}, we derive a slight generalization of the result of An, Wang, and Wong (Corollary \ref{AWW}).

\begin{corollary}
Let $X\subset \mathbb{P}^N$ be a projective variety over $K$ of dimension $n\geq 1$.  Let $D_1,\ldots, D_n$ be hypersurfaces in $\mathbb{P}^N$ such that $X,D_1,\ldots, D_n$ intersect transversally.  Then $X\setminus\cup_{i=1}^nD_i$ is $K$-hyperbolic if $\deg D_i\geq 2$ for $1\leq i\leq n$.
\end{corollary}
\begin{proof}
Assume that there exists a nonconstant analytic map $f:K\to X\setminus\cup_{i=1}^nD_i$.  Then $\frac{m_f(r,D_i)}{\deg D_i}=T_f(r)+O(1)$, $i=1,\ldots, n$.  By Theorem \ref{tmain}, we find that for $r\geq 1$,
\begin{align*}
\sum_{i=1}^n \frac{m_f(r,D_i)}{\deg D_i}=nT_f(r)+O(1)&\leq  \left(n-1+\max_i \frac{1}{\deg D_i}\right)T_f(r)+O(1)\\
&\leq \left(n-\frac{1}{2}\right)T_f(r)+O(1).
\end{align*}
Since $\lim_{r\to\infty}T_f(r)\to \infty$, this is a contradiction.
\end{proof}

Finally, we show the sharpness of inequality \eqref{shin}.

\begin{theorem}
\label{tsharp}
Let $n$ and $d$ be positive integers.  There exist nonsingular hypersurfaces $D_1,\ldots, D_n$ in $\mathbb{P}^n$ of degree $d$, intersecting transversally, and a nonconstant analytic map $f:K\to X$, with image not completely contained in any of the hypersurfaces $D_1,\ldots, D_n$, such that for all $r\geq 1$,
\begin{equation*}
\sum_{i=1}^n \frac{m_f(r,D_i)}{d}=\left(n-1+\frac{1}{d}\right)T_f(r)+O(1).
\end{equation*}
\end{theorem}

\begin{proof}
We prove this for $n\geq 2$, the case $n=1$ being easy.  Consider $\mathbb{P}^n$ with coordinates $(x_0,\ldots, x_n)$.  Let $L$ be the line in $\mathbb{P}^n$ given by $x_2=x_3=\cdots=x_n=0$.  Let $P$ be the point $(1,0,0,\ldots,0)$.  We let $D_1,\ldots, D_n$ be nonsingular hypersurfaces in $\mathbb{P}^n$ of degree $d$, intersecting transversally, such that
\begin{enumerate}
\item  $P\in \cap_{i=1}^nD_i$,
\item  $L\cap D_i=\{P\}$, \quad $i=1,\ldots, n-1$.
\end{enumerate}
Let $f:K\to \mathbb{P}^n$ be defined by $f=(z,1,0,0,\ldots, 0)$.  Then it is immediate from the definitions that for $r\geq 1$,
\begin{equation*}
T_f(r)=\log r.
\end{equation*}
Let $D_i$ be defined by the homogeneous polynomial $Q_i\in K[x_0,\ldots, x_n]$ of degree $d$, $i=1,\ldots, n$.  Since $L\cap D_i=\{P\}$, $i=1,\ldots, n-1$, it follows that $Q_i\circ f$ is a constant for $i=1,\ldots, n-1$.  Then for $r\geq 1$,
\begin{equation*}
\frac{m_f(r,D_i)}{d}=\frac{1}{d}\log \frac{\| f\|_r^d}{|Q\circ f|_r}=\log r+O(1).
\end{equation*}
Since $D_1,\ldots, D_n$ intersect transversally at $P$ and $L\cap D_i=\{P\}$, $i=1,\ldots, n-1$, it follows that $L$ intersects $D_n$ with multiplicity one at $P$.  Then $Q_n\circ f$ is a polynomial of degree $d-1$ and for $r\geq 1$,
\begin{equation*}
\frac{m_f(r,D_n)}{d}=\frac{1}{d}\log \frac{\| f\|_r^d}{|Q\circ f|_r}=\frac{\log r}{d}+O(1).
\end{equation*}
So for $r\geq 1$,
\begin{equation*}
\sum_{i=1}^n \frac{m_f(r,D_i)}{d}=\left(n-1+\frac{1}{d}\right)T_f(r)+O(1)
\end{equation*}
as desired.
\end{proof}

\subsection*{Acknowledgments}

The author would like to thank William Cherry for helpful comments and suggestions.

\bibliography{SMT}
\end{document}